\documentclass[10pt,a4paper,reqno]{amsart}

\usepackage{amssymb,amsmath,mathrsfs}
\usepackage{enumitem}
\usepackage{amscd}
\usepackage{verbatim}
\numberwithin{equation}{section}

\theoremstyle{plain}

\newtheorem{theorem}[subsection]{Theorem}
\newtheorem{proposition}[subsection]{Proposition}
\newtheorem{lemma}[subsection]{Lemma}
\newtheorem{corollary}[subsection]{Corollary}

\theoremstyle{definition}

\newtheorem{definition}[subsection]{Definition}
\newtheorem{remark}[subsection]{Remark}

\newtheorem{example}[subsection]{Example}

\newcommand\F{\mathscr{F}}

\newcommand\eps{\varepsilon}
  \DeclareMathOperator*{\pr}{Pr}%
  \DeclareMathOperator*\cl{cl}%
  \DeclareMathOperator*\epi{epi}%
  \DeclareMathOperator*\dom{dom}%
  \DeclareMathOperator*\id{id}%
\begin{document}

\title[]{A perturbational duality approach in vector optimization}

\author{Sorin-Mihai Grad}
\address{Institute of Mathematics, Leipzig University, Augustusplatz 10, 04109 Leipzig, Germany}
\email{grad@math.uni-leipzig.de} \thanks{SMG was partially supported by the DFG-Project GR3367/4-1.}

\author{Asgar Jamneshan}
\address{Department of Mathematics, ETH Zurich, 8092 Zurich, Switzerland}
\email{asgar.jamneshan@math.ethz.ch} \thanks{AJ gratefully acknowledges financial support through the DFG-Project KU-2740/2-1.}

\begin{abstract}
A perturbational vector duality approach for objective functions $f\colon X\to \bar{L}^0$ is developed, where $X$ is a Banach space and $\bar{L}^0$ is the space of extended real valued functions on a measure space, which extends the perturbational approach from the scalar case. The corresponding strong duality statement is proved under a closedness type regularity condition. Optimality conditions and a Moreau-Rockafellar type formula are provided. The results are specialized for constrained and unconstrained problems. Examples of integral operators and risk measures are discussed.   
\end{abstract}

\maketitle

\setcounter{tocdepth}{1}



\section{Introduction}\label{se1}

In the scalar setting, strong duality for convex optimization problems can be achieved under closedness or interiority type regularity conditions, see e.g.~\cite{BOT,DVO,ZAL}.  The closedness type regularity conditions are viable alternatives to their more restrictive interiority type counterparts, being  successfully applied in subdifferential calculus (see e.g.~\cite{CHJ,BOT, SMG, DVO, DMN}), $DC$ programming (see e.g.~\cite{DMN, FLY}), monotone operators (see e.g.~\cite{SIA, NFA, BOT, JWU}), equilibrium theory (see e.g.~\cite{LAS}) or variational inequalities (see e.g.~\cite{ABW, CCS}). This paper introduces closedness type regularity conditions for the following class of vector optimization problems.    
Let $(X,Y,\langle\cdot,\cdot \rangle)$ be a dual pair of Banach spaces satisfying the following properties:  
\begin{itemize}
\item $|\langle x,y\rangle|\leq \|x\| \|y\|$ for all $x\in X$ and $y\in Y$, 
\item norm-closed balls in $X$ and $Y$ are weakly closed respectively\footnote{
For example, both conditions are satisfied if $Y$ is the norm dual of $X$.}.  
\end{itemize}
Let $(\Omega,\mathcal{F},\mu)$ be a $\sigma$-finite measure space and denote by $\bar{L}^0$ the space of equivalence classes of extended real valued Borel functions on $\Omega$. 
Consider on $\bar{L}^0$ the pointwise almost everywhere order. 
A perturbational approach for conjugate duality as developed by Rockafellar \cite{ROC} for scalar optimization problems is adopted in order to construct a corresponding duality for vector optimization problems of the following typeÖ  
\begin{equation}\label{mainproblem}
\inf f(x), \quad f\colon X\to \bar{L}^0. 
\end{equation} 
To this end a Fenchel-Moreau representation in \cite{DJK} and conditional analysis techniques (see \cite{CKV,DJKK,FKV,JAZ} for an introduction) are employed.  
The main idea in \cite{DJK} is to extend $X$ to the space $L^0(X)$ of equivalence classes of strongly measurable functions on $\Omega$ with values in $X$. 
Then the duality pairing $(X,Y,\langle\cdot,\cdot \rangle)$ is extended to a conditional duality pairing $(L^0(X),L^0(Y),\langle\cdot,\cdot \rangle)$ which induces a conditional weak topology on $L^0(X)$. 
The restriction of this topology to $X$ provides a notion of semi-continuity which allows to extend a function $f\colon X\to \bar{L}^0$ to the larger domain $L^0(X)$. 
In this larger context tools from conditional functional analysis \cite{FKV,JAZ} become applicable which yield a conditional Fenchel-Moreau representation.   
This can be interpreted in the original framework, and thus made fruitful for a perturbational vector duality. 
We are able to establish a corresponding strong duality statement under convexity and semi-continuity hypotheses with a closedness type regularity condition which is accompanied by necessary and sufficient optimality conditions. 
A subdifferential for functions $f\colon X \to \bar{L}^0$ is introduced which is used to characterize these optimality conditions.  
To prove these results a related Moreau-Rockafellar formula is shown. As a byproduct, a Farkas type statement is derived.     

Unconstrained optimization problems with composite objective functions appear in the scalar case for instance in image processing (see e.g.~\cite{BCH,BHS,BHE}), logistics (see e.g.~\cite{BCH}) and machine learning (see e.g.~\cite{BHM,BHS,BLO}), while constrained optimization problems can be found in mathematical economics (see e.g.~\cite{BHR, BLW}), design (see e.g.~\cite{GMW, PBH}) or engineering (see e.g.~\cite{PBH, STM}).   
As special cases of the general problem \eqref{mainproblem}, unconstrained vector optimization problems with composite objective functions and constrained vector optimization problems are studied, for which Moreau-Rockafellar formulas, duality statements and necessary and sufficient optimality conditions are provided. 

As for applications, proper convex functions $f\colon L^p(S,\mathcal{S},\nu)\to \bar{L}^0$, where $(S,\mathcal{S},\nu)$ is a possibly different finite measure space, which satisfy a Fatou continuity property  admit a vectorial Fenchel-Moreau representation. 
This class of functions include nonlinear integral operators and vector-valued and conditional risk measures for which we sketch potential applications.  
See also \cite{ACG,AFM} for related examples of vector optimization problems. 

Several concepts and results from scalar convex analysis and functional analysis are extended to a conditional or $L^0$-module framework in e.g.~\cite{CKV,DJKK,FKV,JAZ}, see also the references therein. 
A direct usage of these results for problem \eqref{mainproblem} is not possible since a Banach space is a priori not an $L^0$-module. 
The setting considered in \cite{DJK} and in the present article require extension results established in \cite{DJK} to take advantage of results in conditional functional analysis. 
As the topological dual space of $L^0$ is trivial in general and the interior of the positive cone in $L^0$ is empty, existing scalarization methods (see e.g.~\cite{GPQ}) cannot be employed. The same applies to set-valued methods or vector-space techniques, see \cite{DJK} for a discussion.
The existing results on vector and set-valued optimization (see \cite{BIS,DVO,SMG,JAH,KTZ} for an overview) cannot be applied neither because of the different frameworks and solution concepts. While in the literature on vector optimization the dual problems are constructed with respect to various efficiency concepts \cite{DVO}, this paper works with pointwise almost everywhere optimality which generates a different duality framework that is a more direct extension of the classical scalar conjugate duality than the existing vector duality approaches. This is also stressed by the fact that the objective function of the dual problem proposed here contains the conjugate of the primal objective function while in the existing vector conjugate duality it usually consists of additionally introduced vectors that have to fulfill certain constraints \cite{DVO}.  
A similar duality approach can be constructed by means of set-valued functions (see e.g.~\cite{DVO}) but these are always accompanied by additional complications that can be avoided in the present approach. Moreover, the considered vector functions can take infinite values at some components and finite otherwise, while in classical vector optimization this is not accepted. 

The remainder of this paper is organized as follows. 
The setting and all relevant notions and results are collected in Section \ref{se2}.  
The perturbational approach to problem \eqref{mainproblem} is introduced in Section \ref{se3}, where the main strong duality statement is proved. 
Examples and specifications of the general result are discussed in Section \ref{sec4}.

\section{Preliminaries}\label{se2}

Throughout, fix a dual pair of Banach spaces $(X,Y,\langle \cdot,\cdot \rangle)$ such that
\begin{itemize}
\item $|\langle x,y\rangle|\leq \Vert x\Vert \Vert y\Vert$ for all $x\in X$ and for all $y\in Y$,
\item and the norm-closed unit balls in $X$ and $Y$ are weakly closed\footnote{Which refers to the initial topologies on $X$ and $Y$ induced by all functionals $\langle \cdot,y\rangle$, $y\in Y$, and all functionals $\langle x,\cdot \rangle$, $x\in X$, respectively.} respectively.
\end{itemize}
An example of such a dual pair is a Banach space paired with its norm dual, see \cite[Section 2]{DJK} for more examples.
Fix also a second dual pair of Banach spaces $(W,Z,\langle \cdot,\cdot\rangle)$ with the same properties. 
Let ${\pr}_W(A)$ denote the projection to $W$ of a set $A\subseteq X\times W$. 
By $\id_X$ we denote the identity operator on $X$. 

Unless specified otherwise, let $(\Omega,\mathcal{F},\mu)$  be a $\sigma$-finite measure space. 
We will identify two measurable sets if their symmetric difference is a null set.
This leads to the associated measure algebra which is a complete Boolean algebra, see \cite[Chapter 31]{GHA} for more details.
Always identify two functions on $\Omega$ if they agree almost everywhere (a.e.).
Let $\bar{L}^0, L^0,L^0_+$ and $L^0_{++}$ denote the spaces of measurable functions on $\Omega$ with values in $[-\infty,+\infty]$, $(-\infty,+\infty)$, $[0,\infty)$ and $(0,\infty)$, respectively.
We always consider on $\bar{L}^0$ the order $s\leq t$ if $s(\omega)\leq t(\omega)$ a.e.
An important property of this order is that it is complete on $\bar{L}^0$ and Dedekind complete if restricted to $L^0$, see e.g.~\cite{FRE}. 
The essential supremum and the essential infimum are denoted by $\sup$ and $\inf$ respectively. 
In particular, any arbitrary family of measurable sets $(A_i)$ in $\mathcal{F}$ admits a supremum and an infimum in $\mathcal{F}$ with respect to a.e.~inclusion.
Following common practice in convex analysis, we stipulate $+\infty + (-\infty) = +\infty$, $0\cdot (+\infty) = +\infty$ and $0\cdot (-\infty) = 0$ (which in $\bar{L}^0$ are understood pointwise a.e.).

For a function $f\colon X\to \bar{L}^0$, \emph{proper} and \emph{convex} are defined by
\begin{itemize}
\item $f(x)>-\infty$ for all $x\in X$ and $f(x_0)<+\infty$ for some point $x_0\in X$,
\item and $f(rx_1+(1-r)x_2)\leq rf(x_1)+(1-r)f(x_2)$ for all $x_1,x_2\in X$ and $r\in[0,1]$, respectively.
\end{itemize}
For a proper function $f\colon X\to \bar{L}^0$, let  ${\dom}(f):=\{x\in X\colon f(x)\in L^0\}$ define its domain and ${\epi}(f):=\{(x,t)\in X\times L^0\colon f(x) \leq t\}$ its epigraph.
The following properties can be directly verified from the definitions.
\begin{itemize}
\item  $f\colon X\to \bar{L}^0$ is convex if and only if $\epi (f)$ is convex.
\item Let $\Phi\colon X\times W\to \bar{L}^0$ be a proper convex function, then $\inf_{w\in W} \Phi(\cdot, w)\colon X\to \bar{L}^0$ is convex.
\end{itemize}
\begin{example}
As a consequence of the previous properties, one can show that the \textit{infimal convolution} $f\square g\colon X\to \bar{L}^0$ of two proper convex functions $f,g\colon X\to \bar{L}^0$ defined by
\[
f\square g (x)=\inf_{p\in X}\{f(p) + g(x-p)\}
\]
is convex as well.  
Analogously, the conditional infimal convolution $F\square G\colon L^0(X)\to \bar{L}^0$ of two proper $L^0$-convex (see definitions below) functions $F,G\colon L^0(X)\to \bar{L}^0$  is defined by 
\[
F\square G (x)=\inf_{p\in L^0(X)}\{F(p) + G(x-p)\}
\]
which is $L^0$-convex.
\end{example}
Our analysis of minimizing a function $f\colon X\to \bar{L}^0$ relies on extensions. 
To this end, we consider different spaces of functions on $\Omega$ with values in $X$. 
Let $L^0_s(X)$ denote the space of step functions $x\colon \Omega\to X$, i.e.~functions whose range is essentially countable. 
Each such function can be represented by $\sum_{n} x_n 1_{A_n}$ where $(x_n)$ is a sequence in $X$ and $(A_n)$ is a measurable partition of $\Omega$, where by $1_A$ we denote the standard characteristic function of a set $A\subseteq \Omega$ defined by $1_A(\omega) = 1$ if $\omega\in A$ and $1_A (\omega)=0$ otherwise. 
We identify $X$ with a subset of $L^0_s(X)$ by the embedding $x\mapsto x 1_\Omega$.
The norm of $X$ can be extended to $L^0_s(X)$ with values in $L^0_+$ via $\|\sum_n x_n 1_{A_n}\|:=\sum_n \|x_n\| 1_{A_n}$. 
Then a function $f\colon X\to \bar{L}^0$ \emph{extends} to $F_s\colon L^0_s(X)\to \bar{L}^0$ by defining $F_s(\sum_n x_n 1_{A_n}):=\sum_n f(x_n)1_{A_n}$. 

Now let $L^0(X)$ denote the space of \emph{strongly measurable} functions $x\colon \Omega\to X$.
Then the norm of $X$ extends to $L^0(X)$ with values in $L^0_+$ by $\|x\|:=\lim_{n\to \infty} \|x_n\|$ where $(x_n)$ is a sequence in $L^0_s(X)$ such that $x_n\to x$ a.e.
Notice that for each $x\in L^0(X)$ and every $t\in L^0_{++}$ there exists $\tilde{x}\in L^0_s(X)$ such that $\|x-\tilde{x}\|<t$.
In order to extend a function $F_s\colon L^0_s(X)\to \bar{L}^0$ to the larger domain $L^0(X)$ (and thus a function $f\colon X\to \bar{L}^0$) a semi-continuity condition is required which is introduced next. 
A function $F\colon L^0(X)\to \bar{L}^0$ is said to be 
\begin{itemize}
\item \emph{local} (or \emph{stable}) if $F(1_A x)=1_A F(1_Ax )$ for all $A\in\mathcal{F}$ and $x\in L^0(X)$, or equivalently $F(\sum_n x_n1_{A_n})=\sum_n F(x_n) 1_{A_n}$ for all measurable partitions $(A_n)$ of $\Omega$ and every sequence $(x_n)$ of $L^0(X)$,  
\item  \emph{$L^0$-linear} if $F(r x_1 + x_2)= r F(x_1) + F(x_2)$ for all $x_1,x_2\in L^0(X)$ and $r\in L^0$, 
\item  \emph{$L^0$-convex} if $F(r x_1 + (1-r) x_2)\leq r F(x_1) + (1-r) F(x_2)$ for all $x_1,x_2\in L^0(X)$ and $r\in L^0$ such that $0\leq r\leq 1$, 
\item  \emph{proper} if  $F(x)>-\infty$ for all $x\in L^0(X)$ and $F(x_0)<+\infty$ for some $x_0\in L^0(X)$. 
\end{itemize}
For a stable and proper function $F\colon L^0(X)\to \bar{L}^0$, let ${\dom}(F):=\{x\in L^0(X)\colon F(x)\in L^0\}$ and ${\epi}(F):=\{(x,t)\in L^0(X)\times L^0\colon F(x) \leq t\}$ be its domain and epigraph respectively. 

A set $H\subseteq L^0(X)\times L^0$ is said to be
\begin{itemize}
\item   \emph{stable} if $H\neq \emptyset$ and $\sum_n (x_n,t_n)1_{A_n}\in H$ for all sequences $(x_n,t_n)$ in $H$ and every measurable partition $(A_n)$ of $\Omega$;
\item  \emph{$L^0$-convex} if $r (x_1,t_1)+ (1-r)(x_2,t_2)\in H$ for all $(x_1,t_1),(x_2,t_2)\in H$ and $\lambda\in L^0$ with $0\leq r\leq 1$.
\end{itemize}
For example ${\dom}(F)$ and ${\epi}(F)$ are stable sets for a proper and stable function $F\colon L^0(X)\to \bar{L}^0$. 
We show next how to construct from a stable set $H$ in $L^0(X)\times L^0$ a function $F_H\colon L^0(X)\to \bar{L}^0$.
For $(x,t)\in L^0(X)\times L^0$, let 
\[
A(x,t):=\sup\{A\in \F\colon (x,t)1_A\in H 1_A\}. 
\]
We show that $A(x,t)$ is attained.
By the properties of the essential supremum, we find a countable sequence $(A_n)$ with $(x,t)1_{A_n} \in H 1_{A_n}$ for each $n$ and $A(x,t)=\cup_n A_n$.
Let $B_n=A_n\cap (\cup_{m<n} A_m)$ for each $n$.
Then $(x,t)1_{B_n} \in H 1_{B_n}$ for every $n$, and $A(x,t)=\cup_n B_n$.
As $H$ is stable, we have $(x,t)1_{A(x,t)}\in H1_{A(x,t)}$.

We define the \emph{stable lower bound function} $F_H\colon L^0(X)\to \bar{L}^0$ by 
\[
F_H(x):=\inf\{t\in L^0\colon (x,t) 1_{A(x,t)}\in H1_{A(x,t)}\}1_{\cup_{(x,t)\in L^0(X)\times L^0} A(x,t)}+(+\infty)1_{(\cup_{(x,t)\in L^0(X)\times L^0} A(x,t))^c}, 
\]
where superscript $c$ denotes complementation. 
The following properties can be verified from the constructions.
\begin{itemize}
\item A proper function $F\colon L^0(X)\to \bar{L}^0$ is stable if and only if $\epi(F)$ is a stable set in $L^0(X)\times L^0$.
\item A proper function $F\colon L^0(X)\to \bar{L}^0$ is $L^0$-convex if and only if $\epi(F)$ is $L^0$-convex.
\item If $F\colon L^0(X)\times L^0(W)\to \bar{L}^0$ is a proper $L^0$-convex function, then $x\mapsto \inf_{w\in L^0(W)} F(x,w)$  is $L^0$-convex.
\item If $H$ is a stable and $L^0$-convex set in $L^0(X)\times L^0$, then $F_H$ is a stable and $L^0$-convex function.
\item Let $F\colon L^0(X)\to \bar{L}^0$ be a stable proper function and $H\subseteq L^0(X)\times L^0$ be a stable set.
Then $F=F_{\text{epi}(F)}$ and $\epi(F_H)=H$.
\end{itemize}
Our notion of semi-continuity stems from so-called conditional topologies.
We refer the interested reader to \cite{DJKK,JAZ} for an introduction to conditional topologies. 
The conditional Euclidean topology on $L^0$, denoted by $\tau$, is given by the following base: 
\[
\{\{a\in L^0\colon |a-b|<t\} \colon t\in L^0_{++}, b\in L^0\}.
\]
The duality pairing on $X\times Y$ extends to $L^0(X)\times L^0(Y)$ with values in $L^0$ by $\langle x , y \rangle:= \lim_{n\to \infty} \langle x_n, y_n\rangle$, where $(x_n)$ is a sequence in $L^0_s(X)$ such that $x_n\to x$ a.e.~and $(y_n)$ is a sequence in $L^0_s(Y)$ such that $y_n\to y$ a.e., and the extension of the duality pairing to $L^0_s(X)\times L^0_s(Y)$ is defined in the natural way.  One has $|\langle x,y\rangle |\leq \|x\|\|y\|$ for all $x\in L^0(X), y\in L^0(Y)$.
By \cite[Lemma 3.1]{DJK}, we also have that $\langle x,y\rangle =0$ for all $x\in L^0(X)$ implies $y=0$, and similarly, $\langle x,y\rangle =0$ for all $y\in L^0(X)$ implies $x=0$.
Thus $(L^0(X),L^0(Y),\langle \cdot,\cdot\rangle)$ defines a conditional dual pair of conditional Banach spaces, see \cite{DJKK,JAZ}.
Let $t\in L^0_{++}$ and $y_1,y_2,\ldots,y_m\in L^0(Y)$. A basic conditional weak neighborhood of $0\in L^0(X)$ is defined by
\[
V^t_{y_1,y_2,\ldots,y_m}:=\cap_{k=1}^m\{x\in L^0(X)\colon |\langle x,y_k\rangle| < t \}.
\]
Let $(A_n)$ be a measurable partition of $\Omega$, $(y_k)_{k=1}^{m_n}$ be a finite sequence in $L^0(Y)$ for each $n$, and $(t_n)$ be a sequence in $L^0_{++}$.
A concatenation of basic neighborhoods is defined by
\[
\{\sum_n x_n 1_{A_n}\colon x_n\in  V^{t_n}_{y_1,y_2,\ldots,y_{m_n}}\}.
\]
The collection of all such concatenations forms a local base of a topology on $L^0(X)$ which will be denoted by $\sigma(L^0(X),L^0(Y))$. 
Then $(L^0(X),\sigma(L^0(X),L^0(Y)))$ will be a topological $L^0$-module (where $L^0$ is endowed with the topology $\tau$), see \cite{JAZ} for a reference. 
The conditional weak topology on $L^0_s(X)$ is defined by relativizing the topology $\sigma(L^0(X),L^0(Y))$ to $L^0_s(X)\subseteq L^0(X)$.
One way to formalize convergence in a topological space is through nets.
By construction, if we restrict attention to nets which respect concatenations, then we do not violate their limiting behavior.
More precisely, if $(x_\alpha)$ is a net in $L^0(X)$, then we suppose that each $\alpha$ is a measurable function on $\Omega$ such that if $(A_n)$ is a measurable partition of $\Omega$ and $(\alpha_n)$ is a sequence of indices such that $\alpha  = \alpha_n$ on $A_n$ for all $n$ for some index $\alpha$, then $x_{\alpha_n} = x_{\alpha}$ on $A_n$ for all $n$.
Such nets exist, see \cite{DJK} and the references therein for details, and they are called \emph{stable nets}.
We are now able to define our notion of semi-continuity.
\begin{itemize}
\item A stable function $F\colon L^0(X)\to \bar{L}^0$ is said to be \emph{$\sigma(L^0(X),L^0(Y))$-lower semi-continuous} if $F(x)\leq \liminf_\alpha F(x_\alpha)$ for every stable net $(x_\alpha)$ such that $x_\alpha\to x$ in the topology $\sigma(L^0(X),L^0(Y))$.
\item A function $f\colon X\to \bar{L}^0$ is said to be \emph{s-lower semi-continuous} if its extension $F_s\colon L^0_s(X)\to \bar{L}^0$ is $\sigma(L^0(X),L^0(Y))$-lower semi-continuous with respect to the relative topology.
\end{itemize}
By inspection, $F$ is stable, proper and $\sigma(L^0(X),L^0(Y))$-lower semi-continuous if and only if $\epi(F)$ is stable and closed.
The  \emph{$\sigma(L^0(X),L^0(Y))$-closure} ${\cl}(F)$ of $F$ is defined by $F_{\cl(\epi(F))}$ where $\cl(\epi(F))$ denotes the closure of $\epi(F)$ in the topology  $\sigma(L^0(X),L^0(Y))$. 
We have $\epi(\cl(F))=\cl(\epi(F))$. 
It can also be verified that if $F$ is stable, proper and $L^0$-convex, then so is $\cl(F)$.
\begin{remark}\label{r:lsc}
Given $F\colon L^0(X)\to \bar{L}^0$, the closure $\cl(F)$ is the largest $\sigma(L^0(X),L^0(Y))$-lower semi-continuous function bounded from above by $F$.
Indeed, since $\epi(\cl(F))=\cl(\epi(F))$ it follows that $\cl(F)$ is $\sigma(L^0(X),L^0(Y))$-lower semi-continuous and it is always less than or equal to $F$ since $\epi (F)\subseteq \epi (\cl (F))$. Taking an arbitrary $\sigma(L^0(X),L^0(Y))$-lower semi-continuous function $G\colon L^0(X)\to \bar{L}^0$ with $G\leq F$, one gets $\epi (F)\subseteq \epi (G)$, followed by $\epi(\cl(F))=\cl(\epi (F))\subseteq \cl(\epi (G))=\epi (G)$. Consequently, $G\leq \cl (F)$.
\end{remark}
We can state the second extension result which was established in \cite{DJK}.
 \begin{theorem}\label{extension}
 Let $f\colon X\to \bar{L}^0$ be s-lower semi-continuous.
 Then there exists a stable $\sigma(L^0(X),L^0(Y))$-lower semi-continuous function $F\colon L^0(X)\to \bar{L}^0$ such that $F|_X=f$.
 Moreover, if $f$ is proper convex, then $F$ is proper $L^0$-convex.
 \end{theorem}
We have the following conditional version of the Fenchel-Moreau theorem stated in \cite{FKV}.
The interested reader is referred to \cite{JAZ} and its references for a background and overview on results in conditional functional analysis.
\begin{theorem}\label{t:FM}
Let $F\colon L^0(X)\to \bar{L}^0$ be a proper $L^0$-convex and $\sigma(L^0(X),L^0(Y))$-lower semi-continuous function.
Then
\[
F(x)=\sup_{y\in L^0(Y)}\{\langle x,y\rangle -F^\ast(y)\},
\]
where the conditional conjugate $F^\ast\colon L^0(Y)\to \bar{L}^0$ is defined by
\[
F^\ast(y)=\sup_{x\in L^0(X)} \{\langle x,y\rangle -F(x)\}.
\]
\end{theorem}
In \cite{DJK}, the conditional version of the Fenchel-Moreau theorem was used to establish a Fenchel-Moreau result for functions $f\colon X\to \bar{L}^0$ as follows.
\begin{theorem}\label{t:FM2}
Let $f\colon X\to \bar{L}^0$ be proper convex and s-lower semi-continuous.
Then
\[
f(x)=\sup_{y\in L^0(Y)}\{\langle x,y\rangle -f^\ast(y)\},
\]
where the \textit{conjugate} $f^\ast\colon L^0(Y)\to \bar{L}^0$ is defined by
\[
f^\ast(y)=\sup_{x\in X} \{\langle x,y\rangle - f(x)\}.
\]
\end{theorem}
\begin{remark}\label{YF}
For any functions $F\colon L^0(X)\to \bar{L}^0$ and $f\colon X\to \bar{L}^0$, the conditional conjugate $F^\ast$ and the conjugate $f^\ast$ (as defined above) are stable, $L^0$-convex and $\sigma(L^0(X),L^0(Y))$-lower semi-continuous. Moreover, the following \emph{Young-Fenchel type inequalities} hold in this framework:
\begin{align*}
F^\ast(y) + F(x) &\geq \langle x,y\rangle\ \forall x \in L^0(X)\, \forall y\in L^0(Y),\\
f^\ast(y) + f(x) &\geq \langle x,y\rangle\ \forall x \in X\, \forall y\in L^0(Y).
\end{align*}
\end{remark}
\begin{example}
Given the proper convex functions $f,g\colon X\to \bar{L}^0$, one has
\[
(f\square g)^\ast (y) = \sup_{x\in X} \{\langle x, y\rangle - \inf_{p\in X}\{f(p) + g(x-p)\}\} =
\sup_{p\in X} \{\langle p, y\rangle - f(p)\} + \sup_{u\in X} \{\langle u, y\rangle - g(u)\} = f^\ast (y) + g^\ast (y)
\]
for all $y\in L^0(Y)$. Similarly, for proper $L^0$-convex functions $F,G\colon L^0(X)\to \bar{L}^0$ it holds $(F\square G)^\ast = F^\ast + G^\ast$.
\end{example}
\begin{lemma}\label{l:lsc}
Let $F\colon L^0(X)\to \bar{L}^0$ be a stable function. Then we have $F^\ast=\cl (F)^\ast$.
\end{lemma}

\begin{proof}
By definition, it follows from $\cl (F) \leq F$ that $F^\ast\leq \cl (F)^\ast$.
In order to prove the opposite inequality, fix $y\in L^0(Y)$ and let
\[
A=\{F^\ast(y)=+\infty\}, B=\{F^\ast(y)=-\infty\}, C=\{-\infty <F^\ast(y)<+\infty\},
\]
and observe that $(A,B,C)$ is a partition of $\Omega$.
Since $F^\ast\leq \cl (F)^\ast$ we have $A=\{\cl(F)^\ast (y)= +\infty\}$, and thus $F^\ast=\cl (F)^\ast$ on $A$.
By the Young-Fenchel type inequality (see Remark \ref{YF}), it holds that $B=\{F(x)=+\infty\}$ for all $x\in L^0(X)$.
Therefore by the local property, we have that $\cl (F)(x) = +\infty$ on $B$ which implies that $\cl (F)^\ast (y)=-\infty$ on $B$ as well.
It remains to show that  $\cl (F)^\ast \leq F^\ast$ holds on $C$.
By the conditional Young-Fenchel inequality, it holds that $\langle x, y\rangle - F^\ast (y) \leq F(x)$ on $C$ for all $x\in L^0(X)$.
Since the function $x\mapsto \langle x, y\rangle - F^\ast (y)$ is a stable $\sigma(L^0(X),L^0(Y))$-lower semi-continuous minorant of $F$ on $C$, it is less than or equal to $\cl (F)$ on $C$. Consequently, $\langle x, y\rangle - \cl(F)(x) \leq F^\ast (y)$ on $C$ for all $x\in L^0(X)$.
Taking the supremum over $x\in L^0(X)$, one obtains
$\cl(F)^\ast (y)\leq F^\ast (y)$ on $C$ which completes this proof.
\end{proof}
We define a subdifferential notion for functions $f\colon X \to \bar{L}^0$.
\begin{definition}\label{d:subdiff}
Let $f\colon X\to \bar{L}^0$ be a function. A dual element $y\in L^0(Y)$ is said to be a \textit{subgradient} of $f$ at $x\in X$ whenever $f(x)\in L^0$ and $f(p)-f(x)\geq \langle p-x,y\rangle$ for all $p\in X$. The set of subgradients of $f$ at $x$ is denoted by $\partial f(x)$ and is said to be the
\textit{subdifferential} of $f$ at $x$. In the case that $f(x)$ is not finite, we take by convention $\partial f(x)=\emptyset$.
\end{definition}
\begin{remark}\label{r:subdiff}
From Definition \ref{d:subdiff} one can derive that, given $y\in L^0(Y)$ and $x\in X$ such that $f(x)\in L^0$, one has $y\in \partial f(x)$ if and only if
$\langle x,y\rangle-f(x)\geq \langle p,y\rangle - f(p)$ for all $p\in X$, that is further equivalent to
$\langle x,y\rangle-f(x)\geq \sup_{p\in X} \{\langle p,y\rangle - f(p)\} = f^\ast(y)$, i.e. $f^\ast(y) + f(x) \leq \langle x,y\rangle$. Taking into consideration the Young-Fenchel type inequality given in Remark \ref{YF}, one concludes that $y\in \partial f(x)$ if and only if $f^\ast(y) + f(x) = \langle x,y\rangle$.
\end{remark}
By $X^*$, we denote the topological dual space of $X$. 
Let $A:X\to W$ be a bounded operator with adjoint $A^\ast\colon W^\ast\to X^\ast$. 
As $A^\ast$ and $A$ are uniformly continuous, by \cite[Proposition 2.8]{DJK18}, there are unique extensions $A\colon L^0(X)\to L^0(W)$ and $A^*:L^0(W^*)\to L^0(X^*)$ such that $\langle A^*(w^*), x\rangle = \langle Ax, w^*\rangle$ for all $x\in L^0(X)$ and $w^*\in L^0(W^*)$. 
In particular, $\langle A^*(w^*), x\rangle = \langle Ax, w^*\rangle$ for all $x\in X$ and $w^*\in L^0(W^*)$. 
A subset $U\subseteq X$ is said to be \textit{s-closed} if $L^0_s(U)$ is closed w.r.t.~the relative $\sigma(L^0(X),L^0(Y))$-topology. 
The (convex) \textit{indicator function} $\delta_U:X\to \bar{L}^0$ of a set $U\subseteq X$ is defined by
\[
\delta_U(x)=
\begin{cases}
0,  &x\in U,\\
\infty 1_\Omega, &x\not\in U. 
\end{cases}
\]
It can be checked that $U$ is non-empty, convex and s-closed if and only if $\delta_U$ is proper convex and s-lower semi-continuous. 
A set $C\subseteq W$ is a \emph{convex cone} if $rC\subseteq C$ for all $r\geq 0$ and $C+C\subseteq C$. 
A convex cone induces a partial ordering ``$\leqq_C$'' on $W$ by $w_1\leqq_C w_2$ whenever $w_2-w_1\in C$. 
By $C^*:=\{z\in Z: \langle w, z\rangle \geq 0 \; \forall w\in C\}$, we denote the \textit{dual cone} of $C$.
Given a convex cone $C\subseteq W$, a function $h : X \to W$ is said to be \textit{$C$-convex} if $r h(x_1)+(1-r)h(x_2) - h(r x_1+(1-r) x_2)\in C$ holds for all $x_1,x_2 \in X$ and all $r \in (0,1)$, and \textit{$C$-$\epi$-closed} whenever its \textit{$C$-epigraph} $\{(x,w)\in X\times W: h(x)-w\in -C\}$ is s-closed.

\section{Main results}\label{se3}

In this section, we develop a duality approach for vector optimization problems consisting in minimizing vector-valued functions defined on Banach spaces and taking values in $\bar{L}^0$ which is inspired by the perturbational approach for conjugate duality established for scalar optimization problems by Rockafellar \cite{ROC}. 
The first result provides a \textit{Moreau-Rockafellar type statement}, see \cite{BOT} for the scalar counterpart.  

\begin{lemma}\label{MR}
Let $\Phi\colon X\times W\to \bar{L}^0$ be a proper convex and s-lower semi-continuous function such that $0\in {\pr}_W({\dom}(\Phi))$.
For all  $y\in L^0(Y)$, it holds that
\begin{equation}\label{eq0}
(\Phi(\cdot,0))^\ast(y)=\sup_{x\in X} \{\langle x,y\rangle - \Phi(x,0)\}={\cl}(\inf_{z\in L^0(Z)} \Phi^\ast(\cdot,z))(y).
\end{equation}
\end{lemma}
\begin{proof}
The key to the proof is the identity
\begin{equation}\label{eq1}			
(\inf_{z\in L^0(Z)}\Phi^\ast(\cdot,z))^\ast(x)=\Phi(x,0), \quad x\in X,
\end{equation}
which is implied by Theorem \ref{t:FM2} as follows
\begin{align*}
(\inf_{z\in L^0(Z)} \Phi^\ast(\cdot,z))^\ast(x)&=\sup_{y\in L^0(Y)}\{\langle x,y\rangle - \inf_{z\in L^0(Z)} \Phi^\ast(y,z)\} \\
&= \sup_{y\in L^0(Y), \, z\in L^0(Z)}\{\langle x,y\rangle - \Phi^\ast(y,z)\} \\
&=  \sup_{y\in L^0(Y), \, z\in L^0(Z)}\{\langle x,y\rangle + \langle 0,z\rangle - \Phi^\ast(y,z)\}\\
&= \Phi(x,0).
\end{align*}
The first equality in \eqref{eq0} is the definition of the conjugate.
By \eqref{eq1}, Lemma \ref{l:lsc} and Theorem \ref{t:FM}, we would have
\[
(\Phi(\cdot,0))^\ast=\big((\inf_{z\in L^0(Z)} \Phi^\ast(\cdot,z))^\ast\big)^\ast={\cl}(\inf_{z\in L^0(Z)} \Phi^\ast(\cdot,z)),
\]
if ${\cl}(\inf_{z\in L^0(Z)} \Phi^\ast(\cdot,z))$ was proper.
Suppose that this is false.
Then we have either
\begin{equation}\label{eq2}
\mu(\{{\cl}(\inf_{z\in L^0(Z)} \Phi^\ast(\cdot,z))(y)=+\infty\})>0
\end{equation}
for all $y\in L^0(Y)$, or
\begin{equation}\label{eq3}
\mu(\{{\cl}(\inf_{z\in L^0(Z)} \Phi^\ast(\cdot,z))(y)=-\infty\})>0
\end{equation}
for some $y\in L^0(Y)$.
Let us deal with \eqref{eq2} first.
We prove that \eqref{eq2} implies the stronger statement
\begin{equation}\label{eq4}		
\mu(\cap_{y\in L^0(Y)}\{{\cl}(\inf_{z\in L^0(Z)} \Phi^\ast(\cdot,z))(y)=+\infty\})>0.
\end{equation}
To see this, suppose for the sake of a contradiction that \eqref{eq4} is false, i.e.~assume\footnote{Note that the possibly uncountable intersection makes sense in the associated measure algebra (see the preliminaries and the references there for more details), and the equality to the empty set is understood in the a.e. sense.}
\[
\cap_{y\in L^0(Y)}\{{\cl}(\inf_{z\in L^0(Z)} \Phi^\ast(\cdot,z))(y)=+\infty\}=\emptyset.
\]
By de Morgan's laws (in a complete Boolean algebra, see \cite{GHA} for a reference), we have thus
\[
\cup_{y\in L^0(Y)}\{{\cl}(\inf_{z\in L^0(Z)} \Phi^\ast(\cdot,z))(y)<+\infty\}=\Omega.
\]
By \cite[Chapter 30, Lemma 1]{GHA}, there exists a countable family $(y_n)$ in  $L^0(Y)$ such that
\[
\cup_n \{{\cl}(\inf_{z\in L^0(Z)} \Phi^\ast(\cdot,z))(y_n)<+\infty\}=\Omega.
\]
Put $A_n=\{{\cl}(\inf_{z\in L^0(Z)} \Phi^\ast(\cdot,z))(y_n)<+\infty\}\setminus(\cup_{m<n} \{{\cl}(\inf_{z\in L^0(Z)} \Phi^\ast(\cdot,z))(y_m)<+\infty\})$ for each $n$.
Then $(A_n)$ forms a partition of $\Omega$.
By the gluing property of the conditional envelope function,
\begin{align*}
\{{\cl}(\inf_{z\in L^0(Z)} \Phi^\ast(\cdot,z))(\sum_n y_n|A_n)<+\infty\}&=\{\sum_n {\cl}(\inf_{z\in L^0(Z)} \Phi^\ast(\cdot,z))(y_n)|A_n<+\infty\}=\Omega.
\end{align*}
But this contradicts \eqref{eq2} for $y=\sum_n y_n|A_n\in L^0(Y)$.
Hence we may continue by assuming \eqref{eq4}.
This implies that ${\cl}(\inf_{z\in L^0(Z)} \Phi^\ast(\cdot,z))^\ast(x)=-\infty$ on a set of positive measure for all $x\in L^0(X)$.
This, however, contradicts the properness of $\Phi$, since Lemma \ref{l:lsc} and \eqref{eq1} imply
\[
{\cl}(\inf_{z\in L^0(Z)}\Phi^\ast(\cdot,z))^\ast=(\inf_{z\in L^0(Z)} \Phi^\ast(\cdot,z))^\ast=\Phi(\cdot,0).
\]
Similarly, the second case \eqref{eq3} implies that $\Phi(x, 0)=+\infty$ on a set of positive measure for all $x\in X$ which contradicts the feasibility assumption $0\in {\pr}_W({\dom}(\Phi))$. This completes the proof.
\end{proof}
The following proposition will prepare for the strong duality statement.
\begin{proposition}\label{MR-epi}
Let $\Phi\colon X\times W\to \bar{L}^0$ be a proper convex and s-lower semi-continuous function such that $0\in {\pr}_W({\dom}(\Phi))$.
Then it holds that
\[
{\epi}((\Phi(\cdot,0))^\ast)={\cl}_{\sigma \times \tau}({\epi}(\inf_{z\in L^0(Z)} \Phi^\ast(\cdot,z)))={\cl}_{\sigma \times \tau}({\pr}_{L^0(Y)\times L^0}({\epi}(\Phi^\ast))).
\]
\end{proposition}
\begin{proof}
The identity on the l.h.s.~ is a consequence of Lemma \ref{MR} and the intertwining relations between epigraphs and closures.
As for the identity on the r.h.s., let  $(y,t)\in {\pr}_{L^0(Y)\times L^0}({\epi}(\Phi^\ast))$.
Then there is $z\in L^0(Z)$ such that $\Phi^\ast(y,z)\leq t$. Thus $\inf_{z\in L^0(Z)} \Phi^\ast(y,z)\leq t$ which implies $(y,t)\in {\epi}(\inf_{z\in L^0(Z)}(\Phi^\ast(\cdot,z)))$.
On the other hand, if $(y,t)\in {\epi}(\inf_{z\in L^0(Z)}(\Phi^\ast(\cdot,z)))$, then for all $\eps\in L^0_{++}$ there is $z\in L^0(Z)$ such that $\Phi^\ast(y,z)\leq t+\eps$.
We have
\[
(y,t+\eps)\in \cup_{z\in L^0(Z)} \epi \Phi^\ast (\cdot, z)={\pr}_{L^0(Y)\times L^0}({\epi}(\Phi^\ast)).
\]
As $\eps\in L^0_{++}$ was arbitrary, one obtains
\[
(y,t)\in {\cl}_{\sigma \times \tau}({\pr}_{L^0(Z)\times L^0}({\epi}(\Phi^\ast))),
\]
which yields the r.h.s.~identity.
\end{proof}
Using these two statements one can prove the following result, whose interpretation in terms of duality will be presented later.
\begin{theorem}\label{stable-d}
Let $\Phi\colon X\times W\to \bar{L}^0$ be a proper convex and s-lower semi-continuous function such that $0\in {\pr}_W({\dom}(\Phi))$.
Then ${\pr}_{L^0(Y)\times L^0}({\epi}(\Phi^\ast))$ is $\sigma(L^0(Y),L^0(X))\times \tau$-closed if and only if
\begin{equation}\label{eq7}
(\Phi(\cdot,0))^\ast (y)=\sup_{x\in X}\{\langle x,y\rangle - \Phi(x,0)\}=\min_{z\in L^0(Z)} \Phi^\ast(y,z)
\end{equation}
for all $y\in L^0(Y)$.
\end{theorem}
\begin{proof}
Let ${\pr}_{L^0(Y)\times L^0}({\epi}(\Phi^\ast))$ be $\sigma(L^0(Y),L^0(X))\times \tau$-closed.
By Proposition \ref{MR-epi}, we have
\[
{\epi}((\Phi(\cdot,0))^\ast)={\epi}(\inf_{z\in L^0(Z)} \Phi^\ast(\cdot,z))={\pr}_{L^0(Y)\times L^0}({\epi}(\Phi^\ast)),
\]
which implies via Lemma \ref{MR}
\[
(\Phi(\cdot,0))^\ast=\sup_{x\in X} \{\langle x,\cdot\rangle - \Phi(x,0)\}=\inf_{z\in L^0(Z)} \Phi^\ast(\cdot,z).
\]
In particular, the infimum in the r.h.s. is attained, that is \eqref{eq7} holds.

Conversely, \eqref{eq7} yields
\[
{\epi}((\Phi(\cdot,0))^\ast)={\pr}_{L^0(Y)\times L^0}({\epi}(\Phi^\ast)).
\]
Proposition \ref{MR-epi} then implies that ${\pr}_{L^0(Y)\times L^0}({\epi}(\Phi^\ast))$ is $\sigma(L^0(Y),L^0(X))\times \tau$-closed.
\end{proof}
We introduce a perturbational vector duality approach for vector optimization problems for $\bar{L}^0$-valued functions defined on $X$.
Let $f: X\to \bar{L}^0$ be proper convex and s-lower semi-continuous. Consider the general vector optimization problem, further referred to as the \emph{primal problem}\\

\noindent $(PG)$\hfill$\inf\limits_{x \in X} f(x)$.\hfill$\:$\\

By $v(PG)$ we denote the optimal objective value of $(PG)$. In order to assign a dual problem to $(PG)$, consider a proper convex and s-lower semi-continuous \textit{perturbation function} $\Phi\colon X\times W\to \bar{L}^0$ fulfilling $\Phi(x,0) = f(x)$ for all $x \in X$. We call $W$ the {\it perturbation space} and its elements {\it perturbation variables}. The problem $(PG)$ can be then rewritten as\\

\noindent $(PG)$\hfill$\inf\limits_{x \in X} \Phi(x,0)$.\hfill$\:$\\

To $(PG)$ we attach the following \textit{conjugate dual problem}\\

\noindent $(DG)$\hfill$\sup\limits_{z \in L^0(Z)} \{-\Phi^\ast(0, z)\}$.\hfill$\:$\\

\noindent For this primal-dual pair of vector optimization problems one can derive directly from the construction the following \textit{weak duality} statement. Note that for the construction of the dual problem and for deriving weak duality the function $\Phi$ needs not be s-lower semi-continuous.

\begin{proposition}\label{weak-d}
It holds $v(DG)\leq v(PG)$, where $v(DG)$ denotes the optimal objective value of the conjugate dual problem.
\end{proposition}

However, of major interest is the situation, where the optimal objective values of the primal and its corresponding dual problem coincide and the dual problem also has optimal solutions, called \textit{strong duality}.
Then Theorem \ref{stable-d} provides the following strong duality statement as a direct consequence.

\begin{theorem}\label{strong-d}
Let $\Phi\colon X\times W\to \bar{L}^0$ be a proper convex and s-lower semi-continuous function such that $0\in {\pr}_W({\dom}(\Phi))$ and
${\pr}_{L^0(Y)\times L^0}({\epi}(\Phi^\ast))$ is $\sigma(L^0(Y),L^0(X))\times \tau$-closed. Then
\begin{equation}\label{eq8}
\inf_{x\in X}\Phi(x,0)=\max_{z\in L^0(Z)} \{-\Phi^\ast(0,z)\}.
\end{equation}
\end{theorem}

\begin{remark}\label{r:stable}
The assertion of Theorem \ref{stable-d} can be seen as a \textit{stable strong duality} statement for the primal-dual pair $(PG)$-$(DG)$, i.e.~for each $y\in L^0(Y)$ there is  strong duality for the primal-dual pairs of vector optimization problems\\

\noindent $(PG_{y})$\hfill$\inf\limits_{x \in X} \{\Phi(x,0) - \langle x, y\rangle\}$,\hfill$\:$\\

\noindent and\\

\noindent $(DG_{y})$\hfill$\sup\limits_{z \in L^0(Z)} \{-\Phi^\ast(y,z)\}$,\hfill$\:$\\

where $(PG_{y})$ was obtained by linearly perturbing the objective function of $(PG)$, while $(DG_{y})$ is its corresponding conjugate dual problem. Thus $(PG)$ is embedded in the family of optimization problems $\{(PG_{y}): y\in L^0(Y)\}$, where it coincides with $(PG_{0})$ and a similar observation is valid for $(DG)$ as well. Note also that, by construction, whenever $y\in L^0(Y)$ one has $v(DG_{y})\leq v(PG_{y})$, i.e. for each of these pairs of primal-dual vector optimization problems there is always weak duality.
\end{remark}

By means of the strong duality statement one can derive necessary and sufficient optimality conditions for the primal-dual pair $(PG)$-$(DG)$.

\begin{corollary}\label{optcond}
Let $\Phi\colon X\times W\to \bar{L}^0$ be a proper convex and s-lower semi-continuous function such that $0\in {\pr}_W({\dom}(\Phi))$.
When $\bar x\in X$ is an optimal solution to the problem $(PG)$ and ${\pr}_{L^0(Y)\times L^0}({\epi}(\Phi^\ast))$ is $\sigma(L^0(Y),L^0(X))\times \tau$-closed, then there exists an optimal solution $\bar z\in L^0(Z)$ of $(DG)$ such that $\Phi (\bar x, 0) + \Phi^\ast (0, \bar z) = 0$. Conversely, given $\bar x\in X$ and $\bar z\in L^0(Z)$ such that $\Phi (\bar x, 0) + \Phi^\ast (0, \bar z) = 0$, then $\bar x$ is an optimal solution to $(PG)$, $\bar z$ one of $(DG)$ and there is strong duality for the primal-dual pair $(PG)$-$(DG)$.
\end{corollary}

\begin{proof}
By Theorem \ref{strong-d}, the existence of an optimal solution $\bar z\in L^0(Z)$ to $(DG)$ such that
$\inf_{x\in X}\Phi(x, 0)=\max_{z\in L^0(Z)}\{-\Phi^\ast(0, z)\}= -\Phi^\ast(0, \bar z)$ is secured. Since $\bar x\in X$ is an optimal solution to the problem $(PG)$, it follows that $\Phi (\bar x, 0) + \Phi^\ast (0, \bar z) = 0$.

Conversely, keeping in mind Proposition \ref{weak-d}, $\Phi (\bar x, 0) + \Phi^\ast (0, \bar z) = 0$ means actually
$$\Phi (\bar x, 0)=\min_{x\in X}\Phi(x, 0)=\max_{z\in L^0(Z)}\{-\Phi^\ast(0, z)\}=-\Phi^\ast (0, \bar z),$$
which implies the desired conclusion.
\end{proof}
\begin{remark}\label{optcond:s}
By Remark \ref{r:subdiff}, the optimality condition $\Phi (\bar x, 0) + \Phi^\ast (0, \bar z) = 0$ given in Corollary \ref{optcond} can be reformulated by means of the subdifferential as $(0, \bar z)\in \partial \Phi (\bar x, 0)$.
\end{remark}
Another consequence of the strong duality statement is the following Farkas type statement. 
\begin{corollary}\label{alt}
Let $\Phi\colon X\times W\to \bar{L}^0$ be a proper convex and s-lower semi-continuous function such that $0\in {\pr}_W({\dom}(\Phi))$ and ${\pr}_{L^0(Y)\times L^0}({\epi}(\Phi^\ast))$ is $\sigma(L^0(Y),L^0(X))\times \tau$-closed. The following statements are equivalent. 
\begin{itemize}
  \item [(i)] $x\in X$ $\Rightarrow$ $\Phi (x, 0) \geq 0$.
  \item [(ii)] $\exists z\in L^0(Z)$: $\Phi^\ast (0, z)\leq 0$. 
\end{itemize}
\end{corollary}

\section{Special cases} \label{sec4}

In this section the general results are specialized for unconstrained vector optimization problems with composite objective functions and for constrained vector optimization problems respectively. At the end of the section examples of nonlinear integral operators and risk measures are discussed.  

\subsection{Unconstrained problems}

Throughout this subsection, consider the duality pairings $(X,X^*,\langle \cdot,\cdot \rangle)$ and $(W,W^*,\langle \cdot,\cdot \rangle)$.
Let $A\colon X\to W$ be a bounded operator, and let $f:X \to \bar{L}^0$ and $g:W \to \bar{L}^0$ be proper convex and $s$-lower semi-continuous  functions fulfilling the feasibility condition $\dom (f) \cap A^{-1}(\dom (g)) \neq \emptyset$. The unconstrained optimization problem\\

\noindent $(PU)$\hfill$\inf\limits_{x \in X} \{f(x) + g(Ax)\}$,\hfill$\:$\\

\noindent is a special case of $(PG)$ by taking the perturbation function $\Phi_U:X\times W\to \bar{L}^0$, $\Phi_U(x, w)= f(x)+g(Ax+w)$, that is proper convex since $f$ and $g$ have the same properties. Moreover, the feasibility condition $0\in {\pr}_W({\dom}(\Phi_U))$ means that there exists an $x\in X$ such that $\Phi_U(x, 0) \in L^0$, i.e. $f(x) + g(Ax)\in L^0$, that happens if and only if $\dom (f) \cap A^{-1}(\dom (g)) \neq \emptyset$. The conjugate $\Phi_U^\ast:L^0(X^*)\times L^0(W^*)\to \bar{L}^0$ of $\Phi_U$ can be computed as
\begin{eqnarray*}
\Phi_U^\ast(x^*, w^*) &=& \sup\limits_{\substack{x\in X,\\ w\in W}} \{\langle x, x^*\rangle + \langle w, w^*\rangle - f(x)-g(Ax+w)\}\\
&=& \sup\limits_{\substack{x\in X,\\ u\in W}} \{\langle x, x^*\rangle + \langle u-Ax, w^*\rangle - f(x)-g(u)\}\\
&=& \sup\limits_{x\in X} \{\langle x, x^*\rangle - \langle x, A^*w^*\rangle - f(x)\} + \sup\limits_{u\in W} \{\langle u, w^*\rangle - g(u)\}\\
&=& f^\ast(x^*-A^*w^*) + g^\ast(w^*).
\end{eqnarray*}
Thus the {\it Fenchel dual} to $(PU)$ turns out to be\\

\noindent $(DU)$\hfill$\sup\limits_{w^* \in L^0(W^*)} \{-f^\ast(-A^*w^*) - g^\ast(w^*)\}$.\hfill$\:$\\

The weak duality statement for the primal-dual pair $(PU)$-$(DU)$ follows by construction (or can be deduced from Theorem \ref{weak-d}).
\begin{theorem}\label{U:wd}
It holds $v(DU)\leq v(PU)$.
\end{theorem}

The other general results in Section \ref{se3} can be specialized for the primal-dual pair $(PU)$-$(DU)$ as follows.

\begin{theorem}\label{U:all}
Let $A\colon X\to W$ be a bounded linear operator, $f\colon X\to \bar{L}^0$ and $g\colon W\to \bar{L}^0$ be proper convex and $s$-lower semi-continuous such that $\dom (f) \cap A^{-1}(\dom (g)) \neq \emptyset$. For each $x^*\in L^0(X^*)$, one has 
$$
(f+g\circ A)^\ast(x^*)={\cl}(\inf_{w^*\in L^0(W^*)} \{f^\ast(\cdot-A^*w^*) + g^\ast(w^*)\})(x^*).
$$
This equality can be refined for each $x^*\in L^0(X^*)$ to
$$
(f + g\circ A)^\ast(x^*)=\min_{w^*\in L^0(W^*)} \{f^\ast(x^*-A^*w^*) + g^\ast(w^*)\}
$$
if and only if $\epi (f^\ast) + (A^*\times \id_{L^0})(\epi (g^\ast))$ is
$\sigma(L^0(X^*),L^0(X))\times \tau$-closed. If this condition is fulfilled,
there is strong duality for the primal-dual pair $(PU)$-$(DU)$, i.e. there exists an optimal solution $\bar w^*\in L^0(W^*)$ to $(DU)$ such that
$$
\inf_{x\in X}\{f(x) + g(Ax)\}=\max_{w^*\in L^0(W^*)} \{-f^\ast(-A^*w^*) - g^\ast(w^*)\}=-f^\ast(-A^*\bar w^*) - g^\ast(\bar w^*),
$$
and, when $\bar x\in X$ is an optimal solution to the problem $(PU)$, the following optimality conditions are fulfilled: 
\begin{enumerate}
  \item [(i)] $f (\bar x) + f^\ast(-A^*\bar w^*) = -\langle A\bar x, \bar w^*\rangle$,
  \item [(ii)] and $g (A\bar x) + g^\ast(\bar w^*) = \langle A\bar x, \bar w^*\rangle$.
\end{enumerate}
Vice versa, given $\bar x\in X$ and $\bar w^*\in L^0(W^*)$ such that $(i)$-$(ii)$ hold, then $\bar x$ is an optimal solution to $(PU)$, $\bar w^*$ one to $(DU)$ and there is strong duality for the primal-dual pair $(PU)$-$(DU)$.
\end{theorem}

\begin{proof}
The first result follows from Lemma \ref{MR}, taking into account the formula of $\Phi_U^\ast$ computed above and the fact that $\Phi_U$ is s-lower semi-continuous since so are $f$ and $g$.

To derive the next equivalence from Theorem \ref{stable-d}, one should note that $(y, t)\in {\pr}_{L^0(X^*)\times L^0}({\epi}(\Phi_U^\ast))$ holds if and only if there is some $w^*\in L^0(W^*)$ such that $f^\ast(x^*-A^*w^*) + g^\ast(w^*) \leq t$, that can be rewritten as $(x^*-A^*w^*, t-g^\ast(w^*))\in \epi f^\ast$, i.e. $(x^*, t) \in \epi (f^\ast) + (A^*\times \id_{L^0})(\epi (g^\ast))$.

The strong duality statement is a consequence of this equivalence (or can be obtained directly from Theorem \ref{strong-d} by taking into account the above calculations).

From Theorem \ref{optcond} one deduces that
$$f (\bar x) + g (A\bar x) + f^\ast(-A^*\bar w^*) + g^\ast(\bar w^*) =0.$$
The optimality conditions $(i)$-$(ii)$ can be derived from Remark \ref{YF}.
\end{proof}

\begin{remark}
From the above assertions one can deduce similar statements for vector optimization problems consisting in minimizing the sum of finitely many functions. These can be obtained also directly from the general case, in which situation it is no longer necessary to consider the duality pairings $(X,X^*,\langle \cdot,\cdot \rangle)$ and $(W,W^*,\langle \cdot,\cdot \rangle)$.
\end{remark}

\begin{remark}
Note also that the optimal objective value of $(DU)$ is actually equal to $-((f^\ast\circ A^*) \square g^\ast)(0)$. 
By Remark \ref{r:stable}, one can deduce that
$\epi (f^\ast) + (A^*\times \id_{L^0})(\epi (g^\ast))$ is $\sigma(L^0(X^*),L^0(X))\times \tau$-closed if and only if $(f+g\circ A)^\ast=(f^\ast\circ A^*) \square g^\ast$ with the infimum in the infimal convolution attained.
\end{remark}

\subsection{Constrained problems}

Let $S \subseteq X$ be a nonempty, convex and $s$-closed set, and let $f : X \to \bar {L}^0$ be proper convex and $s$-lower semi-continuous such that the feasibility condition $\dom (f) \cap  S \neq \emptyset$ is satisfied. The primal problem is\\

\noindent $(PC)$\hfill$\inf\limits_{x \in S} f(x)$.\hfill$\:$\\

A perturbation function that can be employed to assign a dual problem to $(PC)$ as a special case of $(DG)$ is the Fenchel-Lagrange type one (cf. \cite{DVO, BOT}):
$$\Phi_{FL} : X \times X \to \bar{L}^0,\
\Phi_{FL} (x,u) = \left \{ \begin{array}{ll}
f(x+u),& \ \mbox{if} \ x \in S, \\
+\infty,& \ \mbox{otherwise}.
\end{array}\right.$$
It is proper convex and $s$-lower semi-continuous because of the similar properties of $f$ and $S$, and due to the feasibility condition. 
Its conjugate $\Phi_{FL}^\ast : L^0(Y) \times L^0(Y) \to \bar {L}^0$ can be computed as 
\begin{eqnarray*}
\Phi_{FL}^\ast(y,v) &=& \sup \limits_{\substack{x, u\in X, \\ x \in S}} \{\langle x, y\rangle + \langle u, v\rangle - f(x+ u)\}\\
&=& \sup \limits_{\substack{p\in X, x\in S}} \{\langle x, y\rangle + \langle p-x, v\rangle - f(p)\}\\
&=& \sup \limits_{x\in S} \{\langle x, y-v\rangle \} +\sup \limits_{p\in X} \{\langle p, v\rangle  - f(p)\} \\
&=& f^\ast(v) + \delta_S^\ast(y-v).
\end{eqnarray*}
The dual problem it attaches to $(PC)$ is the \textit{Fenchel-Lagrange dual problem}:\\

\noindent $(DC^{FL})$\hfill$\sup\limits_{\substack{v \in L^0(Y)}} \big \{-f^\ast(v) - \delta_S^\ast (-v)\big \}$,\hfill$\:$\\

that can be reformulated as\\

\noindent $(DC^{FL})$\hfill$ -(f^\ast \square \delta_S^\ast)(0)$.\hfill$\:$\\

The weak duality statement for the primal-dual pair $(PC)$-$(DC^{FL})$ follows by construction (or can be deduced from Theorem \ref{weak-d}).
\begin{theorem}\label{FL:wd}
It holds $v(DC^{FL})\leq v(PC)$.
\end{theorem}

The other results proved in the general case in Section \ref{se3} can be specialized for the primal-dual pair $(PC)$-$(DC^{FL})$. 
\begin{theorem}\label{FL:all}
Let the nonempty convex s-closed set $S \subseteq X$ and the proper convex s-lower semi-continuous function $f : X \to \bar {L}^0$ satisfy $\dom (f) \cap  S \neq \emptyset$.
Then for each $y\in L^0(Y)$, one has
$$
(f+\delta_{S})^\ast(y)={\cl}\big(\inf\limits_{\substack{v\in L^0(Y)}} \big\{f^\ast (v) +\delta_S^\ast (y-v)\big\}\big)
= {\cl}\big(\big(f^\ast \square \delta_S^\ast\big) (y)\big).$$
This equality can be refined for each $y\in L^0(Y)$ to
$$
(f+\delta_{S})^\ast(y)=\min\limits_{\substack{v\in L^0(Y)}} \big\{f^\ast(v) +\delta_S^\ast (y-v)\big\}
=\big(f^\ast \square \delta_S^\ast\big)(y),$$
with the infimum in the infimal convolution attained if and only if the set $\epi (f^\ast) + \epi (\delta_S^\ast)$ is $\sigma(L^0(Y),L^0(X))\times \tau$-closed. 
If this condition is fulfilled there is strong duality for the primal-dual pair $(PC)$-$(DC^{FL})$, i.e.~there exists an optimal solution $\bar v\in L^0(Y)$ to $(DC^{FL})$ such that
$$
\inf\limits_{x \in S} f(x)=\max\limits_{v\in L^0(Y)} \big\{-f^\ast (v) -\delta_S^\ast (-v)\big\} =
-f^\ast (\bar v) -\delta_S^\ast (-\bar v),$$
and, when $\bar x\in X$ is an optimal solution to $(PC)$, the following optimality conditions are fulfilled:
\begin{enumerate}
  \item [(i)] $f (\bar x) + f^\ast (\bar v) = \langle \bar x, \bar v\rangle$,
  \item [(ii)] $\delta_S^\ast (-\bar v)= -\langle \bar x, \bar v\rangle$. 
\end{enumerate}
Vice versa, given $\bar x\in X$ and $\bar v\in L^0(Y)$ such that $(i)$-$(ii)$ hold, then $\bar x$ is an optimal solution to $(PC)$, $\bar v$ one to $(DC^{FL})$, and there is strong duality for the primal-dual pair $(PC)$-$(DC^{FL})$.
\end{theorem}

\begin{proof}
The assertions can be deduced from Lemma \ref{MR}, Theorem \ref{stable-d}, Theorem \ref{strong-d} and Theorem \ref{optcond} together with Remark \ref{YF}, respectively, analogously to the proof of Theorem \ref{U:all}, by noting that $(y, t)\in {\pr}_{L^0(Y)\times L^0}({\epi}(\Phi_{FL}^\ast))$ holds if and only if there is $v\in L^0(Y)$ such that $f^\ast (v) +\delta_S^\ast(y-v) \leq t$ that can be rewritten as $(v, f^\ast(v))\in \epi (f^\ast)$ and $(y-v, t-f^\ast(v))\in \epi (\delta_S^\ast)$. Summing these two relations, one gets $(y, t)\in \epi (f^\ast) +  \epi(\delta_S^\ast)$, followed by $(y, t)\in \epi (f^\ast) + \epi (\delta_S^\ast)$.
\end{proof}

\subsection{Examples}\label{se4}

We present two classes of examples to which the results of this article can be applied.
Both classes are of the following general form. Let $(S,\mathcal{S},\nu)$ be a finite measure space, $1<p<\infty$ and $f\colon L^p(S,\mathcal{S},\nu)\to \bar{L}^0(\Omega,\mathcal{F},\mu)$ be a proper convex function. The following \emph{Fatou continuity property} yields the s-lower semi-continuity of $f$, see \cite{DJK18} for a proof.
\begin{itemize}
\item For every sequence $(x_n)$ in $L^p(S,\mathcal{S},\nu)$ such that $\sup_n \|x_n\|_p<\infty$ and $x_n\to x$ a.e., one has $f(x)\leq \liminf f(x_n)$.
\end{itemize}
In this case, $f$ has the representation
\begin{equation}\label{eq:fatou}
f(x)=\sup_{y\in L^0(L^q(S,\mathcal{S},\nu))} \{\langle x,y\rangle - f^\ast(y)\},
\end{equation}
where
\[
f^\ast(y)=\sup_{x\in L^p(S,\mathcal{S},\nu)}\{\langle x,y\rangle - f(x)\}.
\]
This representation is not immediate from Theorem \ref{t:FM2}, see \cite{DJK18} for a proof. The first example is a class of (non)linear integral operators.
\begin{example}
A function $k\colon \Omega\times S\times L^p(S,\mathcal{S},\nu) \to \mathbb{R}$ is said to be a \emph{kernel functional} if the following properties are satisfied:    
\begin{itemize}
\item $k(\omega,s,x)$ is measurable in the product $\Omega\times S$ for all $x\in L^p(S,\mathcal{S},\nu)$, 
\item $\int_S |k(\omega,s,x)|d\nu(s)<\infty$ $\mu$-a.e. $x\in L^p(S,\mathcal{S},\nu)$,  
\item for every sequence $(x_n)$ in $L^p(S,\mathcal{S},\nu)$ such that $\sup_n \|x_n\|_p<\infty$ and $x_n\to x$ a.e., it holds 
\[
k(\omega,\cdot,x)\leq \liminf k(\omega,\cdot,x_n)\quad \mu\text{-a.e.};
\]
\item $k$ is convex in $x$ $\mu$-a.e.   
\end{itemize} 
For a kernel functional $k$, a function $f\colon L^p(S,\mathcal{S},\nu)\to L^0(\Omega,\mathcal{F},\mu)$ defined by
\[
f(x)=\int_S k(\cdot,s,x)d\nu(s)
\]
is called a \emph{(non)linear integral operator}. 
Notice that by Fatou's lemma $f$ satisfies Fatou continuity, and thus is $s$-lower semi-continuous. 
For example, if $k(\omega,s,x)=\tilde{k}(\omega,s)x$ and $\tilde{k}$ is product measurable and satisfies $\int_S |\tilde{k}(\omega,s)|d\nu(s)<\infty$ $\mu$-a.e., then the corresponding integral operator is linear and satisfies all the above properties by the dominated convergence theorem.  
Our vector optimization results could be applied to minimize the sum of two such nonlinear integral operators.   
\end{example}
Our second example are conditional risk measures, see \cite{ACG,ACP,DES,FKF,FRM,FMA} and the references therein for an overview. 
The conditional risk measures associated to a conditional risk acceptance family in \cite{MAG} are defined via a constrained vector optimization problem whose objective vector function takes values in $\bar{L}^0$. 
We consider a definition of a conditional risk measure which is more minimalistic than the one considered in the pertinent literature.
We draw inspiration from the definition provided in \cite{ACG}.
\begin{example}
Let  $(\Omega,\mathcal{F},(\mathcal{F}_t)_{t\geq 0},\mathbb{P})$ be a filtered probability space, $1<p<\infty$ and $t\geq 0$.
A function $\rho_t\colon L^p(\Omega,\mathcal{F},\mathbb{P})\to \bar{L}^0_t:=\bar{L}^0(\Omega,\mathcal{F}_t,\mathbb{P})$ is said to be a \emph{conditional convex risk measure} if the following properties are satisfied:
\begin{itemize}
\item proper convex, 
\item monotonicity, i.e.~$x_1\leq x_2$ implies $\rho_t(x_1)\geq \rho_2(x_2)$, 
\item continuity from above, i.e.~$\rho_t(x_n)\uparrow \rho_t(x)$ a.s.~whenever $x_n\downarrow x$ a.s. 
\end{itemize}
It can be verified that monotonicity and continuity from above imply the Fatou continuity property, see e.g.~\cite[Lemma A.1]{ACG}, or \cite{FOS} for an argument in the static scalar case. Thus we obtain a representation of the form \eqref{eq:fatou}. 
A typical example of a portfolio optimization problem consists in minimizing a risk functional subject to constraints that guarantee the achievement of a certain minimal expected return. 
In the present framework, this can be formulated as a constrained vector optimization problem:\\ 

\hfill$\inf\limits_{x \in S} \rho_t(x),$\hfill$\:$\\

\noindent where $S= \{x \in L^p(\Omega,\mathcal{F},\mathbb{P}): \mathbb{E}[x|\mathcal{F}_t] \geq 0\}$ is $s$-closed. 
\end{example}

\section{Conclusion}\label{se6}

In this paper a new duality approach for vector optimization problems involving objective functions mapping from a Banach space to $\bar {L}^0$ is provided that is a more direct extension of the perturbational one from the scalar case than the existing ones in the literature, by taking advantage  of recent advances in conditional convex analysis. The corresponding strong and stable duality statements, and necessary and sufficient optimality conditions are proven under convexity and topological hypotheses together with closedness type regularity conditions, corresponding Moreau-Rockafellar type formulae being obtained as byproducts. As special cases of the general problem both unconstrained vector optimization problems with composite objective functions and constrained vector optimization problems are worked out. The duality approach is different to the existing ones in the literature, on the one hand because in the constructions conditional analysis is applied, and on the other hand since it covers vector optimization problems for which the classical constructions do not apply. Examples are constructed which illustrate the scope of applications of the obtained theoretical results. 

For future work, we want to complete the proposed duality scheme by providing interiority type regularity conditions for the duality and optimality statements. Although more restrictive than their closedness type counterparts, unlike these, such conditions would be formulated on the underlying  Banach space that could prove to be useful in applications.\\

\end{document}